\documentclass[12pt]{article}

\usepackage[margin=1in]{geometry}       
\usepackage{hyperref}
\newcommand{\eps}{\varepsilon}
\newcommand{\cC}{\mathcal{C}}

\usepackage{amssymb}
\usepackage{amsmath}
\usepackage{amsthm}
\usepackage{mathtools}
\usepackage{a4wide}

\usepackage[numbers,sort&compress]{natbib}

\usepackage{enumitem}

\usepackage[capitalise]{cleveref}
\crefname{equation}{}{}
\crefname{proposition}{Proposition}{Propositions}
\crefname{enumi}{}{}

\usepackage[utf8]{inputenc}

\newtheorem{theorem}{Theorem}[section]
\crefname{theorem}{Theorem}{Theorems}
\newtheorem{lemma}[theorem]{Lemma}
\newtheorem{proposition}[theorem]{Proposition}
\newtheorem{corollary}[theorem]{Corollary}
\newtheorem{conjecture}[theorem]{Conjecture}

\newtheorem{claim}[theorem]{Claim}

\theoremstyle{definition}
\newtheorem{definition}[theorem]{Definition}

\numberwithin{equation}{section}

\DeclareMathOperator{\blue}{blue}
\DeclareMathOperator{\red}{red}

\newenvironment{proofclaim}[1][Proof of Claim]{\begin{proof}[#1]}{\end{proof}}

\usepackage[UKenglish]{babel}
\usepackage[UKenglish]{isodate}

\newcommand{\EMAIL}[1]{\textit{{E-mail}}: \texttt{\href{mailto:#1}{#1}}} 

\title{On $k$-uniform tight cycles: the Ramsey number for $C_{kn}^{(k)}$ and an approximate Lehel's conjecture}

\author{Vincent Pfenninger \thanks{Institute of Discrete Mathematics, Graz University of Technology, Austria, \EMAIL{pfenninger@math.tugraz.at}}}

\date{\today}

\begin{document}

\maketitle

\begin{abstract}
    A \emph{$k$-uniform tight cycle} is a $k$-graph with a cyclic ordering of its vertices such that its edges are precisely the sets of~$k$ consecutive vertices in that ordering.
    We show that, for each $k \geq 3$, the Ramsey number of the $k$-uniform tight cycle on $kn$ vertices is $(1+o(1))(k+1)n$. This is an extension to all uniformities of previous results for $k = 3$ by Haxell, {\L}uczak, Peng, R\"odl, Ruci\'nski, and Skokan and for $k = 4$ by Lo and the author and confirms a special case of a conjecture by the former set of authors.

    Lehel's conjecture, which was proved by Bessy and Thomass\'e, states that every red-blue edge-coloured complete graph contains a red cycle and a blue cycle that are vertex-disjoint and together cover all the vertices. We also prove an approximate version of this for $k$-uniform tight cycles. 
    We show that, for every $k \geq 3$, every red-blue edge-coloured complete $k$-graph on $n$ vertices contains a red tight cycle and a blue tight cycle that are vertex-disjoint and together cover $n - o(n)$ vertices.
\end{abstract}

\section{Introduction}

For a $k$-graph ($k$-uniform hypergraph) $H$, we denote by $r(H)$ the \emph{Ramsey number} of~$H$, that is, the smallest integer $N$ such that every $2$-edge-colouring of the complete $k$-graph on $N$ vertices contains a monochromatic copy of $H$.\footnote{For any $2$-edge-colouring we will assume that the colours are red and blue.}

The Ramsey number for cycles in graphs was determined in \cite{Bondy1973,Faudree1974,Rosta1973} to be, for $n \geq 5$,
\begin{align*}
    r(C_n) = 
    \begin{cases}
        \frac{3}{2}n -1, \quad\text{ if $n \in 2 \mathbb{N}$}, \\
        2n-1,  \quad\text{ if $n \in 2 \mathbb{N} +1$}. 
    \end{cases}
\end{align*}
There are several possible extensions of cycles to $k$-graphs. Here we will focus on \emph{tight cycles}. For the Ramsey number of \emph{loose cycles}, we refer the reader to \cite{Shahsiah2018} and the references therein.
We denote by $C_n^{(k)}$ the \emph{$k$-uniform tight cycle on $n$ vertices}, that is the $k$-graph with vertex set $V(C_n^{(k)}) = \{v_1, \dots, v_n\}$ and edge set $E(C_n^{(k)}) = \{v_i \dots v_{i+k-1} \colon i \in [n]\}$ where the indices are taken modulo $n$. For $k = 3$, Haxell, {\L}uczak, Peng, R\"odl, Ruci\'nski, and Skokan \cite{Haxell2007, Haxell2009} determined $r(C_n^{(3)})$ asymptotically.\footnote{All asymptotics in this paper are considered as $n \rightarrow \infty$ (and all other parameters are fixed).}
\begin{theorem}[{Haxell, {\L}uczak, Peng, R\"odl, Ruci\'nski, and Skokan \cite{Haxell2007, Haxell2009}}]
    \begin{align*}
        r(C_{3n+i}^{(3)}) =
        \begin{cases}
            (1+o(1))4n, \quad\text{ if $i = 0$}, \\
            (1+o(1))6n,  \quad\text{ if $i \in \{1,2\}$}. 
        \end{cases}
    \end{align*}
\end{theorem}
The same authors also made a conjecture for the asymptotics of $r(C_n^{(k)})$ for all $k \geq 3$.
\begin{conjecture}[{Haxell, {\L}uczak, Peng, R\"odl, Ruci\'nski, and Skokan \cite{Haxell2007, Haxell2009}}] \label{HLPRRS_conjecture}
    For $k \geq 3$, $0 \leq i \leq k-1$, and $d \coloneqq \gcd(k,i)$, we have
    \[
        r(C_{kn+i}^{(k)}) = (1+o(1))\frac{d+1}{d}kn.
    \]
\end{conjecture}
The lower bound for \cref{HLPRRS_conjecture} is given by the following construction. Let $N = \frac{d+1}{d}kn -2$ and colour the edges of $K_N^{(k)}$ as follows.\footnote{We denote by $K_N^{(k)}$ the complete $k$-graph on $N$ vertices.} Partition the vertices into two sets $X$ and $Y$ with $|X| = \frac{k}{d}n -1$ and $|Y| = kn -1$. Then colour all edges which intersect $X$ in an even number of vertices red and all other edges blue. For this colouring, there is no monochromatic copy of $C_{kn+i}^{(k)}$ (see \cite[Proposition 7.2]{Lo2025}).

Lo and the author \cite{Lo2025} proved that $r(C_{4n}^{(4)}) = (1+o(1))5n$, confirming \cref{HLPRRS_conjecture} in the case where $k = 4$ and $i = 0$.
In this paper, we prove \cref{HLPRRS_conjecture} in the case where $i = 0$ for all $k \geq 3$.
\begin{theorem} \label{thm:k_unif_ramsey}
    For $k \geq 3$ and $\eps >0$, there exists a positive integer $n_0 = n_0(k, \eps)$ such that if $n \geq n_0$, then $r(C_{kn}^{(k)}) \leq (k+1 + \eps)n$.
\end{theorem}

We remark that, as a straightforward corollary from this, we also obtain the asymptotics of the Ramsey number for the $k$-uniform tight path on $n$ vertices. We define the \emph{$k$-uniform tight path} $P_n^{(k)}$ as the $k$-graph with vertex set $\{v_1, \dots, v_{n}\}$ and edge set $\{v_i \dots v_{i+k-1} \colon i \in [n-k+1]\}$.
\begin{corollary}
    For $k \geq 3$, we have $r(P_n^{(k)}) = (1+o(1))\frac{k+1}{k}n$.
\end{corollary}
Indeed, the upper bound follows immediately from \cref{thm:k_unif_ramsey}, since $P_n^{(k)}$ is a sub-$k$-graph of the $k$-uniform tight cycle on $k\lceil \frac{n}{k} \rceil$ vertices. The lower bound is given by the same extremal example as the one for \cref{HLPRRS_conjecture} with $i = 0$.

Essentially the same proof as the one for \cref{thm:k_unif_ramsey} also gives an approximate Lehel's conjecture for $k$-uniform tight cycles. Lehel conjectured that the vertices of every $2$-edge-coloured complete graph can be partitioned into a red cycle and a blue cycle (where any set of at most $2$ vertices is also considered a cycle). This was proved for large $n$ by Łuczak, Rödl, and Szemer{\'e}di~\cite{Luczak1998}. The bound on $n$ was later improved by Allen~\cite{Allen2008}. Subsequently, Bessy and Thomass\'e~\cite{Bessy2010} gave a proof for all $n$. The analogous problem for tight cycles in hypergraphs was first considered by Bustamante, H\`an, and Stein \cite{Bustamante2017} who showed that an approximate Lehel's conjecture holds for $3$-uniform tight cycles.
\begin{theorem}[{Bustamante, H\`an, and Stein \cite{Bustamante2017}}]
    Every $2$-edge-coloured complete $3$-graph contains a red tight cycle and a blue tight cycle that are vertex-disjoint and together cover $n - o(n)$ vertices.
\end{theorem}
This was improved upon by Garbe, Mycroft, Lang, Lo, and Sanhueza-Matamala~\cite{Garbe2024}. 

\begin{theorem}[{Garbe, Mycroft, Lang, Lo, and Sanhueza-Matamala~\cite{Garbe2024}}] \label{thm:Garbe}
    There exists $n_0 \in \mathbb{N}$ such that for every $n \geq n_0$ and every $2$-edge-coloured complete $3$-graph on $n$ vertices, we have the following. 
    \begin{enumerate}[label= \upshape{(\roman*)}]
        \item There are a red tight cycle and a blue tight cycle that are vertex-disjoint and together cover at least $n - 2$ vertices. 
        \item There are two monochromatic tight cycles (possibly of the same colour) that are vertex-disjoint and together cover all the vertices. \label{GMLLSM_ii}
    \end{enumerate}
\end{theorem}
An example was given in \cite{Garbe2024} which shows the necessity of allowing the two monochromatic tight cycles in \cref{thm:Garbe} \cref{GMLLSM_ii} to possibly have the same colour. This example was generalised to all uniformities $k \geq 3$ in \cite{Lo2020}. Therefore, for $k \geq 3$, partitioning the vertices of every $2$-edge-coloured complete $k$-graph into a red tight cycle and a blue tight cycle is not possible. 

Lo and the author \cite{Lo2020} showed that an approximate Lehel's conjecture still holds for $4$-uniform tight cycles and gave a weaker result for $5$-uniform tight cycles.
\begin{theorem}[{Lo and Pfenninger \cite{Lo2020}}]
    We have the following.
    \begin{enumerate}[label= \upshape{(\roman*)}]
        \item Every $2$-edge-coloured complete $4$-graph on $n$ vertices contains a red tight cycle and a blue tight cycle that are vertex-disjoint and together cover $n-o(n)$ vertices.
        \item Every $2$-edge-coloured complete $5$-graph on $n$ vertices contains $4$ monochromatic tight cycles that are vertex-disjoint and together cover $n-o(n)$ vertices.
    \end{enumerate}
\end{theorem}

In subsequent work, Lo and the author \cite{Lo2024} showed that for a $2$-edge-coloured $k$-graph, $k$ vertex-disjoint monochromatic tight cycles suffice to cover almost all vertices.

\begin{theorem}[{Lo and Pfenninger \cite{Lo2024}}]
    Every $2$-edge-coloured complete $k$-graph on $n$ vertices contains~$k$ vertex-disjoint monochromatic tight cycles covering $n-o(n)$ vertices.
\end{theorem}

In this paper, we improve on this result by showing that an approximate Lehel's conjecture holds for $k$-uniform tight cycles for all $k \geq 3$. 
\begin{theorem} \label{thm:k_unif_Lehel}
    For $k \geq 3$ and $\eps > 0$, there exists a positive integer $n_0 \coloneqq n_0(k, \eps)$ such that if $n \geq n_0$, then every $2$-edge-coloured complete $k$-graph on $n$ vertices contains a red tight cycle and a blue tight cycle that are vertex-disjoint and together cover at least $(1-\eps)n$ vertices.
\end{theorem}

\subsection{Proof Sketch}

We now give a sketch of the proof of \cref{thm:k_unif_ramsey} and \cref{thm:k_unif_Lehel}. The novel idea of this paper is to combine a method from \cite{Lo2025} (on iteratively finding fractional matchings) with a lemma about the structure of monochromatic tight components in `almost complete' red-blue edge-coloured $k$-graphs from \cite{Lo2024} (that allows us to find these fractional matchings). 

We use a, by now standard, method called \emph{{\L}uczak's connected matching method} (introduced in \cite{Luczak1999}).
Roughly speaking, the idea of the method is to apply Szemer\'edi's Regularity Lemma (or a hypergraph version thereof) and then find a matching in the reduced graph that is `connected'  (where the appropriate notion of `connectivity' depends on the application). This `connected matching' is then converted into a cycle in the original graph (or hypergraph).  
This means that the problem of finding a cycle in the original graph is reduced to the problem of finding a `connected matching', that is a matching using only edges from a `connected component', in the reduced graph. To make this more precise, we need to define the right notion of a `connected component' for our setting which is the notion of a \emph{monochromatic tight component}. This is defined as follows. Let $H$ be a $k$-graph. We say that a sequence of edges $e_1 \dots e_t$ in $H$ is a \emph{tight pseudo-walk (from $e_1$ to $e_t$)} if for every $i \in [t-1]$, we have $|e_i \cap e_{i+1}| \geq k-1$.\footnote{For technical reasons, we allow $|e_i \cap e_{i+1}| = k$, that is, $e_i = e_{i+1}$.} A $k$-graph $H$ is called \emph{tightly connected} if for any two edges $f, f' \in E(H)$, there exists a tight pseudo-walk from $f$ to $f'$ in $H$. A \emph{tight component of $H$} is a maximal \emph{tightly connected} sub-$k$-graph of $H$. If $H$ is $2$-edge-coloured, then the \emph{monochromatic tight components of $H$} are the tight components of $H^{\red}$ and $H^{\blue}$ (which are also called \emph{red tight components} and \emph{blue tight components}), respectively.\footnote{We denote by $H^{\red}$ the sub-$k$-graph of $H$ induced by the red edges of $H$ and define $H^{\blue}$ analogously.}

The hypergraph version of {\L}uczak's connected matching method we use (see \cref{cor:matching_to_cycle}) now states that in order to prove \cref{thm:k_unif_ramsey} it suffices to show that every $2$-edge-coloured `almost complete' $k$-graph $H$ on $n$ vertices contains a monochromatic tight component $T$ and a matching in $T$ of size at least $\frac{n}{k+1}$ (in fact a fractional matching of the same weight suffices). 

We find this matching as follows. Assume for simplicity that $H$ is complete. In actuality, we will have to work with a $k$-graph that is only `almost complete'.\footnote{We will define later what exactly we mean by `almost complete'.} However, this will not complicate the argument much. 
We start by showing that there is a monochromatic tight component in $H$ that contains a matching $M$ of size at least $c_k n$, for some small constant $c_k$ depending only on $k$ (see \cref{lem:large_MTC}). Let us assume that this monochromatic tight component is red and denote it by $R$. If $M$ is not yet large enough, we show that a larger (fractional) matching in a monochromatic tight component can be found. We then iterate this process to find larger and larger (fractional) matchings that are contained in monochromatic tight components. To avoid having to deal with increasingly complex fractional matchings, we convert fractional matchings back to integral ones by moving to a blow-up of $H$ (noting that a matching in a blow-up of $H$ can be converted into a fractional matching in $H$ of the corresponding weight, as we show in  \cref{prop:matching_to_fractional}). 

Since it is possible that there is no larger matching in a red tight component, we might have to switch the colour of our matching. To that end, we first show that we can either find a larger (fractional) matching in $R$ or we can find a matching $M'$ in a blue tight component with $|M'| = |M|$ (see \cref{lem:matching_inc_1}). We assign each edge $e \in M$ a distinct vertex $v_e$ not covered by $M$.
If for some $e \in M$ all the edges in $H[e \cup \{v_e\}]$ are red, then $H[e \cup \{v_e\}] \subseteq R$ and we get a larger fractional matching by giving weight $1/k$ to each of the $k+1$ edges in $H[e \cup \{v_e\}]$. It follows that for all $e \in M$, there is a blue edge in $H[e \cup \{v_e\}]$. This gives us a matching $M'$ of blue edges in $H$ with $|M'| = |M|$. We would like to now conclude that all the edges in $M'$ are in the same blue tight component of $H$ (unless there is a larger (fractional) matching in $R$). To show that this is indeed the case, we use a crucial lemma about the structure of the monochromatic tight components in $H$ that Lo and the author introduced in \cite{Lo2024}. This lemma (see \cref{lem:blue_tight_walk}) implies that if there is a tight pseudo-walk $e_1 \dots e_t$ where $e_1$ and $e_t$ are in different blue tight components and $e_2, \dots, e_{t-1}$ are all in the red tight component $R$, then any tight pseudo-walk from $e_1$ to $e_t$ contains an edge of $R$. 
This gives us a powerful tool to find edges in particular monochromatic tight components. Suppose there are two edges $e', f' \in M'$ that are in different blue tight components of $H$. Let $e$ and $f$ be the unique edges in $M$ such that $|e \cap e'| = k-1$ and $|f \cap f'|= k-1$. Since $e$ and $f$ are both in the red tight component $R$, there exists a tight pseudo-walk $P$ from $e$ to $f$. Applying the lemma to the tight pseudo-walk $e' P f'$ gives us that every tight pseudo-walk from $e'$ to $f'$ contains an edge of $R$. This allows us to find a larger (fractional) matching in $R$ (see \cref{claim:R_small_fractional} in the proof of \cref{lem:matching_inc_1}). Thus if we assume that no larger (fractional) matching in $R$ exists, we can conclude that all edges in $M'$ are in the same blue tight component $B$ of $H$.

We now enlarge one of $M$ and $M'$ as follows (see \cref{lem:matching_inc_2}). Pick an edge $g$ of $H$ which is disjoint from $V(M) \cup V(M')$. Note that $|V(M) \cup 
V(M')| = (k+1)|M|$, so assuming $|M| < \frac{n}{k+1}$ there are just enough vertices to pick $g$. Assume without loss of generality that~$g$ is red. We may assume that $g \notin R$ (else we get a larger matching in $R$ by adding~$g$ to~$M$). Let $e \in M$ and let $e' \in M'$ be the unique edge with $|e \cap e'| = k-1$. Consider a tight pseudo-walk $P^*$ from $e$ to $g$ with $V(P^*) \subseteq e \cup g$. Since $e$ and $g$ are in different red tight components of $H$, we have that~$P^*$ contains a blue edge. Let $f$ be the first blue edge on~$P^*$. If $f \notin B$, then the tight pseudo-walk obtained by starting with $e'$ and then following $P^*$ from $e$ to $f$ allows us to again apply \cref{lem:blue_tight_walk} to get a larger fractional matching in~$R$. Hence we may assume that $f \in B$. Generating more edges of $B$ in this way allows us to obtain a larger fractional matching in $B$.

To prove \cref{thm:k_unif_Lehel}, we proceed similarly by first constructing a large matching $M$ in a monochromatic tight component in the same way. Assume without loss of generality that $M$ is in a red tight component $R$. By the argument above, we can now assume that $|M| \geq \frac{n}{k+1}$ (possibly replacing $H$ by a blow-up of it). Let $W \coloneqq V(H) \setminus V(M)$ and note that $|W| \leq \frac{n}{k+1}$. Assign each vertex $w \in W$ to a unique edge $e_w \in M$. As before we can show that unless we can move to a larger (fractional) matching in $R$, there exists a blue tight component $B$ of $H$ such that for each $w \in W$, we have that $H[e_w \cup \{w\}]$ contains an edge of $B$. Since $e_w \cup \{w\}$ is a set of $k+1$ vertices that contains a red edge of $R$ and a blue edge of $B$, any edge in $H[e_w \cup \{w\}]$ is in $R$ or $B$. We let $\varphi$ be the fractional matching which gives weight $1/k$ to all $k+1$ edges of $H[e_w \cup \{w\}]$ for each $w \in W$ and weight $1$ to each edge in $M \setminus \{e_w \colon w \in W\}$. We have that $\varphi$ has weight $\frac{n}{k}$ and only gives non-zero weight to edges in $R$ or $B$. The hypergraph version of {\L}uczak's connected matching method we use (\cref{cor:matching_to_cycle}) implies that this is enough to prove \cref{thm:k_unif_Lehel}.

\section{Preliminaries}

For a positive integer $k$, we let $[k] = \{1, \dots, k\}$ and for a set $V$, we let $\binom{V}{k}$ be the set of subsets of $V$ of size $k$. We sometimes abuse notation by writing $x_1 \dots x_j$ for $\{x_1, \dots, x_j\}$.

A \emph{$k$-graph} $H$, is a pair of sets $(V(H),E(H))$ such that $E(H) \subseteq \binom{V(H)}{k}$. We call the elements of $V(H)$ \emph{vertices} and the elements of $E(H)$ \emph{edges}. We abuse notation and identify $H$ with its edge set $E(H)$. In particular, by $|H|$ we mean the number of edges of $H$. For $W \subseteq V(H)$, we denote by $H[W]$ the $k$-graph with $V(H[W]) = W$ and $E(H[W]) = \{e \in E(H) \colon e \subseteq W\}$.
For $S \subseteq V(H)$ with $1 \leq |S| \leq k-1$, we let $N_H(S) \coloneqq \{S' \in \binom{V(H)}{k-|S|} \colon S' \cup S \in H\}$ and $d_H(S) \coloneqq |N_H(S)|$.

A \emph{matching} in a $k$-graph is a set $M$ of pairwise disjoint edges of $H$. For a matching $M$ in $H$, we denote by $V(M)$ the set of vertices covered by $M$.
A \emph{fractional matching} in $H$ is a function $\varphi \colon H \rightarrow [0,1]$ such that for each $v \in V(H)$, $\sum_{e \in H, \, v \in e} \varphi(e) \leq 1$. For each edge $e \in H$, we call $\varphi(e)$ the \emph{weight of $e$ given by $\varphi$}. The \emph{weight of $\varphi$} is $\sum_{e \in H} \varphi(e)$. For a matching $M$ in $H$, the \emph{fractional matching induced by $M$} is the fractional matching $\varphi \colon H[V(M)] \rightarrow [0,1]$ with $\varphi(e) = 1$ if $e \in M$ and $\varphi(e) =0$ otherwise.
For $r \in \mathbb{N}$, $\varphi$ is called \emph{$1/r$-fractional}, if $\varphi(e) \in \{0, 1/r, 2/r, \dots, (r-1)/r, 1\}$ for all $e \in H$.
Let $H_1$ and $H_2$ be $k$-graphs with $V(H_1) \cap V(H_2) = \varnothing$. We define $H_1 \cup H_2$ to be the $k$-graph with $V(H_1 \cup H_2) = V(H_1) \cup V(H_2)$ and $E(H_1 \cup H_2) = E(H_1) \cup E(H_2)$. 
Let $\varphi_1$ and $\varphi_2$ be fractional matchings in $H_1$ and $H_2$, respectively. Then we define $\varphi_1 + \varphi_2 \colon H_1 \cup H_2 \rightarrow [0,1]$ to be the fractional matching in $H_1 \cup H_2$ with $(\varphi_1 +\varphi_2)(e) = \varphi_1(e)$ if $e \in H_1$ and $(\varphi_1 +\varphi_2)(e) = \varphi_2(e)$ if $e \in H_2$. Let $H$ be a $k$-graph and $H'$ be a sub-$k$-graph of $H$. For a fractional matching $\varphi$ in $H'$, we define the \emph{completion of $\varphi$ with respect to $H$} to be the fractional matching $\varphi^H$ in $H$ with $\varphi^H(e) = \varphi(e)$ for $e \in H'$ and $\varphi^H(e) = 0$ for $e \in H \setminus H'$. If $H$ is a $2$-edge-coloured $k$-graph and $T$ is a monochromatic tight component of $H$, then we abuse notation and do not distinguish between fractional matchings in $T$ and fractional matchings in $H$ which give non-zero weight only to edges in $T$. We call both of these \emph{fractional matchings in $T$}.

We use the \emph{hierarchy notation}, that is when we say that a statement holds for $0 < a \ll b \leq 1$ this is a short form of stating that there exists a non-decreasing function $a_0 \colon (0,1] \rightarrow (0,1]$ such that the statement holds for all $a,b \in (0,1]$ with $a \leq a_0(b)$ (so the statement holds provided that $a$ is sufficiently small in terms of $b$). Hierarchies with more variables are defined similarly. Moreover, if $1/a$ appears in a hierarchy we implicitly assume that $a$ is a positive integer. 

We omit floors and ceilings whenever doing so does not affect the argument. 

\section{The Structure of Monochromatic Tight Components and Blow-ups}

The following lemma is a result from earlier work of Lo and the author \cite{Lo2024}. It gives crucial structural information about the monochromatic tight components in a $2$-edge-coloured `almost complete' $k$-graph. 
To state the lemma, 
we now define our notion of `almost complete'. We say that a $k$-graph is \emph{$(\mu, \alpha)$-dense} if for every $i \in [k-1]$ and all but at most $\alpha \binom{n}{i}$ sets $S \in \binom{V(H)}{i}$, we have $d_H(S) \geq \mu \binom{n}{k-i}$ and $d_H(S) = 0$ for all other $S \in \binom{V(H)}{i}$. 
We think of a $k$-graph $H$ that is $(1-\eps, \eps)$-dense, for some small $\eps >0$, as `almost complete'.
One reason this is a convenient notion of `almost complete' is that if we have an edge $e \in H$, then this means that any subset $S \subsetneq e$ is contained in almost as many edges as if $H$ was complete (since $S$ is contained in $e$, we have $d_H(S) > 0$ and thus $d_H(S) \geq (1-\eps) \binom{n}{k-|S|}$ since $H$ is $(1-\eps, \eps)$-dense).
Moreover, we define a \emph{closed tight pseudo-walk} in a $k$-graph $H$ as a tight pseudo-walk $e_1 \dots e_t$ with $|e_1 \cap e_t| \geq k-1$. 

\begin{lemma}[{\cite[Lemma 3.6]{Lo2024}}] \label{lem:blue_tight_walk}
    Let $1/n \ll \eps \ll 1/k \leq 1/2$. Let~$H$ be a $2$-edge-coloured $(1-\eps, \eps)$-dense $k$-graph on~$n$ vertices. Let $Q= e_1 \dots e_m$ be a closed tight pseudo-walk in~$H$ such that for some $1 < i < m$, we have that~$e_1$ and~$e_i$ are edges in different blue tight components of~$H$. Then there exists a red edge in $\{e_2 \dots, e_{i-1}\}$ that is in the same red tight component as a red edge in $\{e_{i+1}, \dots, e_m\}$. The statement also holds with colours reversed.
\end{lemma}

We will make use of \cref{lem:blue_tight_walk} several times to find more edges of a given monochromatic tight component. To define the closed tight pseudo-walks to apply \cref{lem:blue_tight_walk}, it will be convenient to use the notion of a \emph{tight pseudo-walk induced by a vertex sequence}. For a sequence $x_1 \dots x_t$ of vertices in a $k$-graph $H$, we define the \emph{tight pseudo-walk induced by the vertex sequence $x_1 \dots x_t$} as the tight pseudo-walk $e_1 \dots e_{t-k+1}$ where for $i \in [t-k+1]$, $e_i = x_i \dots x_{i+k-1}$.

As we mentioned in the proof sketch, we use blow-ups to avoid some of the difficulties of dealing with fractional matchings. We make sure that our fractional matchings always give rational weights to edges. So they are $1/r$-fractional matchings, for some $r \in \mathbb{N}$, which means they correspond to matchings in an $r$-blow-up of the original $k$-graph. To that end, let us define $r$-blow-ups formally. 

\begin{definition}[$r$-blow-up] \label{def:blow_up}
    For $k$-graphs $H_*$ and $H$, we say that $H_*$ is an $r$-blow-up of $H$ if there exists a partition $\{V_x \colon x \in V(H)\}$ of $V(H_*)$ into parts of size $r$ such that 
    \[
        H_* = \bigcup_{x_1\dots x_k \in H} K_{V_{x_1}, \dots, V_{x_k}},
    \]
    where $K_{V_{x_1}, \dots, V_{x_k}}$ is the complete $k$-partite $k$-graph with parts $V_{x_1}, \dots, V_{x_k}$. Moreover, if $H$ is $2$-edge-coloured, we require that $H_*$ is $2$-edge-coloured and
    \[
        H_*^{\red} = \bigcup_{x_1\dots x_k \in H^{\red}} K_{V_{x_1}, \dots, V_{x_k}} \quad \text{ and } \quad H_*^{\blue} = \bigcup_{x_1\dots x_k \in H^{\blue}} K_{V_{x_1}, \dots, V_{x_k}}.
    \]
    Denote by $\pi \colon H_* \rightarrow H$, the map such that for $e_* \in H_*$, $\pi(e_*) = x_1 \dots x_k \in H$ is the unique edge such that $e_* \in K_{V_{x_1}, \dots, V_{x_k}}$. Note that, if $H$ is $2$-edge-coloured, then $\pi$ induces a map, denoted by $\pi_{\mathrm{MTC}}$, from the set of monochromatic tight components $\mathcal{C}(H_*)$ of $H$ to the set of monochromatic tight components $\mathcal{C}(H)$ of $H$, where for $T \in \mathcal{C}(H_*)$, we have $\pi_{\mathrm{MTC}}(T) = \{\pi(e)\colon e \in T\} \in \mathcal{C}(H)$. In fact, it is easy to see that $\pi_{\mathrm{MTC}}$ is a bijection and for each $T \in \mathcal{C}(H_*)$, the monochromatic tight component $\pi_{\mathrm{MTC}}(T)$ has the same colour as~$T$.
\end{definition} 

The following proposition gives the fact that having a $1/(rr')$-fractional matching in a $k$-graph $H$ of weight $\mu |V(H)|$ is equivalent to having a $1/r'$-fractional matching in the $r$-blow-up $H_*$ of $H$ of weight $\mu |V(H_*)|$. Moreover, this equivalence behaves well with respect to monochromatic tight components.

\begin{proposition}[{see also \cite[Proposition 3.3]{Lo2025}}] \label{prop:matching_to_fractional}
    Let $r, r' \geq 1$ and $k \geq 2$. Let $H$ be a $2$-edge-coloured $k$-graph and let $H_*$ be an $r$-blow-up of $H$. Let $\cC_*$ be a set of monochromatic tight components in $H_*$ and let $\cC \coloneqq \{\pi_{\mathrm{MTC}}(T_*) \colon T_* \in \cC_*\}$ be the corresponding set of monochromatic tight components in $H$.
    Then the following are equivalent.
    \begin{enumerate}[label = \upshape{(F\arabic*)}, leftmargin= \widthof{F100000}]
        \item There exists a $1/r'$-fractional matching $\varphi_*$ in $H_*$ of weight $\mu |V(H_*)|$ such that $\{e \in H_* \colon \varphi_*(e) > 0\} \subseteq \bigcup \cC_*$. \label{F1}

        \item There exists a $1/(rr')$-fractional matching $\varphi$ in $H$ of weight $\mu |V(H)|$ such that $\{e \in H \colon \varphi(e) > 0\} \subseteq \bigcup \cC$. \label{F2}
    \end{enumerate}
\end{proposition}
\begin{proof}
    We show that \cref{F1} implies \cref{F2}. We omit the proof of the other direction as it is analogous.
    Given $\varphi_*$, we let $\varphi \colon H \rightarrow [0,1]$ be the fractional matching with 
    \begin{align*}
        \varphi(e) \coloneqq \frac{1}{r} \sum_{f \in \pi^{-1}(e)} \varphi_*(f).
    \end{align*}
    To check that $\varphi$ indeed takes values in $[0,1]$, fix $e \in H$ and $x \in e$. We have
    \begin{align*}
        \varphi(e) = \frac{1}{r} \sum_{f \in \pi^{-1}(e)} \varphi_*(f) = \frac{1}{r} \sum_{y \in V_x} \sum_{f \in \pi^{-1}(e),\, y \in f} \varphi_*(f) \leq \frac{1}{r} \sum_{y \in V_x} \sum_{f \in H_*, \, y \in f} \varphi_*(f) \leq 1,
    \end{align*}
    since $|V_x| = r$ and for each $y \in V_x$, $\sum_{f \in H_*, \, y \in f} \varphi_*(f) \leq 1$ as $\varphi_*$ is a fractional matching in $H_*$.
    To check that $\varphi$ is indeed a fractional matching in $H$, note that, for $x \in V(H)$, we have
    \begin{align*}
        \sum_{\substack{e \in H, \, x \in e}} \varphi(e) = \frac{1}{r} \sum_{\substack{e \in H, \, x \in e}} \sum_{f \in \pi^{-1}(e)} \varphi_*(f) = \frac{1}{r} \sum_{f \in H_*, \, f \cap V_x \neq \varnothing} \varphi_*(f) = \frac{1}{r} \sum_{y \in V_x} \sum_{f \in H_*, \, y \in f} \varphi_*(f) \leq 1.
    \end{align*}
    Since $\varphi_*$ is a $1/r'$-fractional matching, $\varphi$ is a $1/(rr')$-fractional matching.
    Note that the weight of $\varphi$ is
    \begin{align*}
        \sum_{e \in H} \varphi(e) = \frac{1}{r}\sum_{e \in H} \sum_{f \in \pi^{-1}(e)} \varphi_*(f) = \frac{1}{r} \sum_{f \in H_*} \varphi_*(f) = \frac{1}{r} \mu |V(H_*)| = \mu |V(H)|. 
    \end{align*}

    If $\varphi_*$ only gives non-zero weight to edges in $\bigcup \cC_*$, then by the definition of $\varphi$, we have that $\varphi$ only gives non-zero weight to edges in $\bigcup \cC$. 
\end{proof}

The following proposition states that being `almost complete' is preserved under taking a blow-up. 

\begin{proposition}[{\cite[Proposition 3.2]{Lo2025}}]
\label{prop:H_*_is_dense}
Let $1/n \ll \eps, \alpha, 1/r, 1/k$ and $k \geq 2$. Let~$H$ be a $(1-\eps, \alpha)$-dense $k$-graph on~$n$ vertices and let~$H_*$ be an $r$-blow-up of~$H$. Then~$H_*$ is $(1- 2\eps, 2 \alpha)$-dense. 
\end{proposition}

\section{The Initial Monochromatic Tight Component}

In this section, we show that a $2$-edge-coloured `almost complete' $k$-graph $H$ has a monochromatic tight component which contains a matching of linear size. This will be the starting point from which we will then show that we can construct larger and larger matchings. 

\begin{lemma} \label{lem:large_MTC}
    Let $1/n \ll \eps \ll 1/k \leq 2$. Let $H$ be a $2$-edge-coloured $(1-\eps, \eps)$-dense $k$-graph on $n$ vertices. Then $H$ contains a monochromatic tight component $T$ that contains a matching of size at least $\frac{n}{k^2 15^k}$.
\end{lemma}

We derive this from the fact that every $k$-graph which has a constant fraction of all possible edges, contains a tight component which still has a constant fraction of all possible edges.

\begin{proposition} \label{prop:large_TC}
    Let $1/n \ll \alpha, 1/k < 1$. Let $H$ be a $k$-graph on $n$ vertices. If $|H| \geq \alpha \binom{n}{k}$, then $H$ contains a tight component $T$ with $|T| \geq \left(\frac{\alpha}{5}\right)^k \binom{n}{k}$. 
\end{proposition}

To prove this, it is convenient to introduce the notion of the `shadow' of a $k$-graph. 
For a $k$-graph $H$, we define its shadow $\partial H$ as the $(k-1)$-graph with $V(\partial H) = V(H)$ and $E(\partial H) = \{e \in \binom{V(H)}{k-1} \colon e \subseteq f \text{ for some } f \in E(H)\}$. We use the following theorem which is called the Lov\'asz version of the Kruskal-Katona theorem and gives a lower bound on the size of the shadow.

\begin{theorem}[Lov\'asz 1979 \cite{Lovasz1979}] \label{thm:Lovasz_Kruskal_Katona}
    Let $k \geq 2$ and let $H$ be a $k$-graph. If $|H| \geq \binom{x}{k}$ for some $x \geq 0$, then $|\partial H| \geq \binom{x}{k-1}$.\footnote{Here we define $\binom{x}{k} = \frac{1}{k!} \prod_{i =1}^k (x-i+1)$ for any $x \in \mathbb{R}$ and $k \in \mathbb{N}$.}
\end{theorem}

We are now ready to prove \cref{prop:large_TC}.

\begin{proof}[Proof of \cref{prop:large_TC}]
    Let $\beta \coloneqq (\alpha/5)^k$.
    Suppose for a contradiction that all tight components in $H$ have size less than $\beta \binom{n}{k}$. Then we can partition $H = H_0 \dot\cup H_1 \dot\cup \dots \dot\cup H_\ell$ (allowing $H_0 = \varnothing$) such that $|H_0| < \beta \binom{n}{k}$, and for each $i \in [\ell]$, $H_i$ is a union of tight components of $H$ with $\beta \binom{n}{k} \leq |H_i| < 2\beta \binom{n}{k}$.\footnote{Indeed, such a partition can be obtained as follows. Start with each tight component as its own set. Then, as long as there are two sets of size less than $\beta \binom{n}{k}$, replace those two sets by their union.} Note that 
    \begin{align} \label{eq:shadows_disjoint}
        \partial H_i \cap \partial H_j = \varnothing \quad \text{ for all } \quad 1 \leq i < j \leq \ell
    \end{align} 
    since $H_i$ and $H_j$ consist of distinct tight components. Let $i \in [\ell]$. Note that $|H_i| \geq \beta \binom{n}{k} \geq \binom{\beta^{1/k}n}{k}$. By \cref{thm:Lovasz_Kruskal_Katona}, $|\partial H_i| \geq \binom{\beta^{1/k}n}{k-1}$. Hence, by \cref{eq:shadows_disjoint}, we have
    \begin{align*}
        \binom{n}{k-1} \geq \sum_{i \in [\ell]} |\partial H_i| \geq \ell \binom{\beta^{1/k}n}{k-1} \geq \frac{\ell \beta^{\frac{k-1}{k}}}{2} \binom{n}{k-1}. 
    \end{align*}
    Thus $\ell \leq 2 \beta^{-\frac{k-1}{k}}$. It follows that 
    \begin{align*}
        (\alpha - \beta)\binom{n}{k} < |H| - |H_0| = \sum_{i \in [\ell]} |H_i| \leq 2 \ell \beta \binom{n}{k} \leq 4 \beta^{1/k} \binom{n}{k} < (5\beta^{1/k} - \beta) \binom{n}{k}.
    \end{align*}
    Since $\beta = (\alpha/5)^k$, this is a contradiction.
\end{proof}

We can now derive \cref{lem:large_MTC} from \cref{prop:large_TC}.

\begin{proof}[Proof of \cref{lem:large_MTC}]
    Note that since $H$ is $(1-\eps, \eps)$-dense, we have 
    \begin{align*}
        |H| \geq \frac{1}{k}\sum_{S \in \binom{V(H)}{k-1}} d_H(S) \geq \frac{1}{k} (1-\eps)^2\binom{n}{k-1} n \geq (1-\eps)^2 \binom{n}{k}.
    \end{align*}
    Hence we may assume without loss of generality that 
    \begin{align*}
        |H^{\red}| \geq \frac{1}{2}(1-\eps)^2 \binom{n}{k} \geq \frac{1}{3} \binom{n}{k}.
    \end{align*}
    By \cref{prop:large_TC}, we have that $H^{\red}$ contains a tight component $T$ with $|T| \geq \frac{1}{15^k} \binom{n}{k}$. Since $T$ is a red tight component in $H$, it now suffices to show that $T$ contains a matching of size at least $\frac{n}{k^2 15^k}$. Let $M$ be a maximal matching in $T$. Since every edge of $T$ intersects at least one edge in $M$, we have
    \begin{align*}
        \frac{1}{15^k} \binom{n}{k} \leq |T| \leq |M| \, k \binom{n-1}{k-1}.
    \end{align*}
    It follows that $|M| \geq \frac{n}{k^2 15^k}$.
\end{proof}

\section{Larger and Larger Matchings}

In this section, we show how to find successively larger (fractional) matchings in a $2$-edge-coloured `almost complete' $k$-graph that are contained in monochromatic tight components.

We first show that given a matching $M$ in a red tight component $R$ (that is not yet large enough), we can either find a larger (fractional) matching in $R$ or we can find a blue tight component that contains a matching which is about as large as $M$ and intersects $M$ in a `simple' way. 

\begin{lemma} \label{lem:matching_inc_1}
    Let $1/n \ll \eps \ll \gamma \ll \eta \ll 1/k$ with $k \geq 3$. Let $H$ be a $2$-edge-coloured $(1-\eps, \eps)$-dense $k$-graph on $n$ vertices. Let $M$ be a matching in a red tight component $R$ of $H$ with $\eta n \leq |M| \leq (1 - \eta) n/k$. Then at least one of the following holds. 
    \begin{enumerate}[label = \upshape{(M\arabic*)}, leftmargin= \widthof{M100000}]
        \item There is a $1/k!$-fractional matching in $R$ of weight at least $|M| + \gamma n$. \label{condition_1}
        \item There is a matching $M'$ in a blue tight component of $H$ with 
        \[
            |M'| \geq \min(|M|, n - |V(M)|)  - \sqrt{\gamma} n
        \]
        such that each edge $e \in M'$ intersects exactly one edge $f(e) \in M$ and $|e \cap f(e)| = k-1$ and $H[e \cup f(e)] \cong K_{k+1}^{(k)}$. \label{condition_2}
    \end{enumerate}
    The same statement holds with colours reversed.
\end{lemma}

The next lemma shows that if our matchings are not yet large enough and we are in the situation arising from \cref{condition_2}, then we can extend either the red matching or the blue matching to get a larger (fractional) matching.

\begin{lemma} \label{lem:matching_inc_2}
    Let $1/n \ll \eps \ll \delta \ll \eta \ll 1/k$ with $k \geq 3$. Let $H$ be a $2$-edge-coloured $(1-\eps, \eps)$-dense $k$-graph on $n$ vertices. Let $R$ and $B$ be a red and a blue tight component of $H$, respectively. Let $\eta n \leq t \leq \left(1 - \eta\right)\frac{n}{k+1}$. Let $M$ be a matching in $R$ and let $M'$ be a matching in $B$ with $|M| = |M'| = t$ such that there exists a bijection $g \colon M \rightarrow M'$ such that for each $f \in M$, $f \cap V(M') \subseteq g(f)$, $|f \cap g(f)| = k-1$, and $H[f \cup g(f)] \cong K_{k+1}^{(k)}$.
    Then there exists a $1/k!$-fractional matching in $R$ or in $B$ of weight at least $t + \delta n$.
\end{lemma}

\subsection{\texorpdfstring{Proof of \cref{lem:matching_inc_1}}{Proof of the First Main Lemma}}

\begin{proof}[Proof of \cref{lem:matching_inc_1}]
    Let $W \coloneqq V(H) \setminus V(M)$. Since $|M| \leq (1 - \eta)n/k$, we have $|W| \geq \eta n$. Let $M^* \subseteq M$ with $|M^*| = \min(|M|, |W|)  - \gamma n$. Let $e \in M^*$ and note that, since $H$ is $(1-\eps, \eps)$-dense, we have
    \[
        \left|W \cap \bigcap_{S \in \binom{e}{k-1}} N_H(S) \right| \geq |W| - k \eps n \geq |M^*|.
    \]
    Hence we may choose distinct vertices $w_e \in W \cap \bigcap_{S \in \binom{e}{k-1}} N_H(S)$ for each $e \in M^*$. Let $M_0 \coloneqq \{e \in M^* \colon H^{\red}[e \cup w_e] \cong K^{(k)}_{k+1}\}$. Note that if $|M_0| \geq k \gamma n$, then we are done since the $1/k!$-fractional matching $\varphi \colon R \rightarrow [0,1]$ with $\varphi(e) = 1$ for each $e \in M \setminus M_0$, $\varphi(f) = 1/k$ for each $f \in \bigcup_{e \in M_0} H[e \cup w_e]$, and $\varphi(f) = 0$ otherwise has weight at least $|M| + \gamma n$ (and thus satisfies \cref{condition_1}). Thus we may assume $|M_0| \leq k \gamma n$. 
    
    Let $M^{**} \coloneqq M^* \setminus M_0$. Note that 
    \begin{align} \label{M^**_size}
    \eta n/2 \leq \min(|M|, |W|) - (k+1)\gamma n \leq |M^{**}| \leq |M^*| = \min(|M|, |W|) - \gamma n.
    \end{align}
    For each $e \in M^{**}$, choose a blue edge $f_e \in H^{\blue}[e \cup w_e]$ (which exists since $e \notin M_0$). Let $M_1 \coloneqq \{ f_e \colon e \in M^{**}\}$. Consider the auxiliary graph $G$ with $V(G) = M_1$ and for $f, f' \in M_1$, $ff' \in E(G)$ if and only if $f$ and $f'$ are in different blue tight components of~$H$.
    \begin{enumerate}[label=\textbf{Case \Alph*:}, ref=\Alph*, wide, labelwidth=0pt, labelindent=0pt]
        \item \textbf{\boldmath The graph $G$ contains a matching $M_G$ of size $k\gamma n$.\unboldmath} If follows from the fact that $H$ is $(1-\eps, \eps)$-dense and $|W| \geq \eta n$ that we can choose disjoint sets $W_F \in \binom{W}{k-1}$ for each $F = f_1f_2 \in M_G$ such that $H[f_1 \cup W_F] \cong H[f_2 \cup W_F] \cong K_{2k-1}^{(k)}$. 
        Let $F = f_1f_2 \in M_G$. We denote by $e_1$ and $e_2$ the unique edges in $M^*\setminus M_0$ such that $|e_1 \cap f_1| = |e_2 \cap f_2| = k-1$. 
        \begin{claim} \label{claim:R_small_fractional}
            There exists a $1/k!$-fractional matching $\varphi_F \colon R[e_1 \cup f_1 \cup e_2 \cup f_2 \cup W_F] \rightarrow [0,1]$ of weight at least $2 + 1/k$.
        \end{claim}
        \begin{proofclaim}
            Let $W_F = \{w_1, \dots, w_{k-1}\}$. Let $(x_1, x_2) \in f_1 \times f_2$ and fix arbitrary orders $f_1 \setminus x_1 = u_1 \dots u_{k-1}$ and $f_2 \setminus x_2 = v_1 \dots v_{k-1}$. Let $P_{x_1, x_2}$ be the tight pseudo-walk induced by the vertex sequence $x_1 u_1 \dots u_{k-1} w_1 \dots w_{k-1} v_1 \dots v_{k-1}x_2$ (note that $P_{x_1,x_2}$ starts in $f_1$ and ends in $f_2$). Observe that since $e_1, e_2 \in R$, there exists a tight pseudo-walk $P^*$ in $R$ from $e_2$ to $e_1$. Since $f_1$ and $f_2$ are in different blue tight components, by applying \cref{lem:blue_tight_walk} to the closed tight pseudo-walk obtained by concatenating $P_{x_1, x_2}$ and $P^*$, we deduce that $P_{x_1, x_2}$ contains an edge $f_{x_1, x_2} \in R$. 
            Note that, by the way $P_{x_1, x_2}$ is constructed, for each $(x_1, x_2) \in f_1 \times f_2$, the edge $f_{x_1, x_2}$ intersects either $f_1$ or $f_2$ but not both and moreover 
            \begin{align} \label{f_nonintersect}
                f_{x_1, x_2} \cap \{x_1, x_2\} = \varnothing.
            \end{align}
            Let $x_1 \in f_1$ and consider the set of edges $E_{x_1} = \{f_{x_1, x_2} \colon x_2 \in f_2\}$. If all edges in $E_{x_1}$ intersect $f_2$, then we obtain the desired $1/k!$-fractional matching $\varphi_F$ by setting $\varphi_F(e_1) = 1$, $\varphi_F(f) = 1/|E_{x_1}|$ for all $f \in \{e_2\} \cup E_{x_1}$, and $\varphi_F(f) = 0$ otherwise.
            Note that by \cref{f_nonintersect}, we have $\bigcap (\{f_2\} \cup E_{x_1}) = \varnothing$. Since $e_2 \cap V(P_{x_1,x_2}) \subseteq f_2$ for each $(x_1, x_2) \in f_1 \times f_2$, we obtain $\bigcap (\{e_2\} \cup E_{x_1}) = \varnothing$.
            This implies that $\varphi_F$ is indeed a fractional matching as no vertex is contained in all $|E_{x_1}|+1$ edges in $\{e_2\} \cup E_{x_1}$.
            Hence we may assume that for every $x_1 \in f_1$ there exists $\sigma(x_1) \in f_2$ such that $f_{x_1, \sigma(x_1)}$ intersects $f_1$ (and thus does not intersect $f_2$). Let $E = \{f_{x_1, \sigma(x_1)} \colon x_1 \in f_1\}$. We obtain the desired $1/k!$-fractional matching $\varphi_F$ by setting $\varphi_F(e_2) =1$, $\varphi_F(f) = 1/|E|$ for each $f \in \{e_1\} \cup E$ and $\varphi_F(f) = 0$ otherwise (noting that $\bigcap (\{e_1\} \cup E) = \varnothing$ by \cref{f_nonintersect}).
        \end{proofclaim}
        Now let $\widetilde{M} \coloneqq (M\setminus M^{**}) \cup \{e \in M^{**} \colon f_e \notin V(M_G)\}$. Let $\varphi_0$ be the fractional matching induced by the matching $\widetilde{M}$. Now the completion of $\varphi_0 + \sum_{F \in M_G} \varphi_F$ with respect to $R$ is a $1/k!$-fractional matching in $R$ of weight at least $|M| + |M_G|/k = |M| + \gamma n$ (and thus satisfies \cref{condition_1}).

        \item \textbf{\boldmath The graph $G$ does not contain a matching of size $k\gamma n$.\unboldmath}

        It follows that $G$ contains an independent set $M'$ of size at least $|V(G)| - 2k\gamma n$. Note that $M' \subseteq M_1$ is a matching in $H^{\blue}$ and since it is an independent set in $G$ it is contained in a blue tight component $B$ of $H$ (by the definition of $G$). Observe that $|M'| \geq |V(G)| - 2k\gamma n = |M^{**}| - 2k\gamma n \geq \min(|M|, |W|) - (3k +1) \gamma n \geq \min(|M|, n - |V(M)|) -\sqrt{\gamma}n$ by \cref{M^**_size} and the definition of $W$. 
        Hence we are done since $M'$ satisfies \cref{condition_2}.\qedhere
    \end{enumerate}
\end{proof}

\subsection{\texorpdfstring{Proof of \cref{lem:matching_inc_2}}{Proof of the Second Main Lemma}}

\begin{proof}[Proof of \cref{lem:matching_inc_2}]
    Let $W \coloneqq V(H) \setminus (V(M) \cup V(M'))$ and note that $|W| \geq n - (k+1)t \geq \eta n$. Since $H$ is $(1-\eps, \eps)$-dense, it follows that there exists a matching $M_0$ in $H[W]$ of size $8k\delta n$ and an injection $f \colon M_0 \rightarrow M$ such that for each $e \in M_0$, $H[e \cup f(e) \cup g(f(e))] \cong K_{2k+1}^{(k)}$. 
    Indeed, it is easy to see that for an edge $f \in M$, we can greedily choose vertices $w_1, \dots, w_k \in W$ one by one such that for each $i \in [k]$, $H[w_1 \dots w_i \cup f \cup g(f)] \cong K_{k+1+i}^{(k)}$ (so that we can then set $f(e) = f$ for $e = w_1 \dots w_k$). 
    Doing this for $8k\delta n$ many different edges in $M$ (using new vertices in $W$ each time) gives the desired matching $M_0$ in $H[W]$ and the desired injection $f \colon M_0 \rightarrow M$.
    Let $M_1 \subseteq M_0$ be a monochromatic matching of size $4k\delta n$. Assume without loss of generality that $M_1$ is red. If there exists $M_2 \subseteq M_1$ such that $M_2 \subseteq R$ and $|M_2| = 2k\delta n$, then we are done since the $1/k!$-fractional matching induced by $M_2 \cup M$ has weight $t + 2k\delta n \geq t + \delta n$. Hence we may assume that there exists $M_3 \subseteq M_1$ with $|M_3| = 2k\delta n$ such that no edge of $M_3$ is in $R$. 
    For $e \in M_3$, recall that $g(f(e)) \in M' \subseteq B$ is the unique edge that intersects $f(e) \in M \subseteq R$. 
    Let $e \in M_3$.
    We fix arbitrary orders $e = w_1(e) \dots w_k(e) \eqqcolon w_1 \dots w_k$ and $f(e) \cap g(f(e)) = z_1(e) \dots z_{k-1}(e) \eqqcolon z_1 \dots z_{k-1}$. Let $x \coloneqq x(e)$ be the unique vertex in $f(e) \setminus g(f(e))$ and let $y \coloneqq y(e)$ be the unique vertex in $g(f(e)) \setminus f(e)$. For each $i \in [k-1]$, consider the tight pseudo-walk $P_i(e)$ induced by the vertex sequence $z_ixz_1 \dots z_{i-1}z_{i+1} \dots z_{k-1} w_1 \dots w_k$. Since $e =w_1 \dots w_k$ and $f(e) = xz_1\dots z_{k-1}$ are in different red tight components of $H$, $P_i(e)$ contains a blue edge. Let $g_i(e)$ be the first blue edge on $P_i(e)$. For each $e \in M_3$, we have either
    \begin{enumerate}
        \item $g_i(e) \in B$ for all $i \in [k-1]$ or \label{e_property_1}
        \item $g_{i^*}(e) \notin B$ for some $i^* \in [k-1]$. \label{e_property_2}
    \end{enumerate}
    Let $M_4 \subseteq M_3$ with $|M_4| = k\delta n$ be such that every $e \in M_4$ satisfies \cref{e_property_1} or every $e \in M_4$ satisfies \cref{e_property_2}.
    \begin{enumerate}[label=\textbf{Case \Alph*:}, ref=\Alph*, wide, labelwidth=0pt, labelindent=0pt]
    \item \textbf{\boldmath Every $e \in M_4$ satisfies \cref{e_property_1}.\unboldmath} Let $e \in M_4$ and let $F_e = \{g_i(e) \colon i \in [k-1]\} \subseteq B$ and note that $\bigcap (\{g(f(e))\} \cup F_e) = \varnothing$ since $g_i(e) \cap \{y(e), z_i(e)\} = \varnothing$ for all $i \in [k-1]$. Let $\varphi_e \colon B[e \cup f(e) \cup g(f(e))] \rightarrow [0,1]$ be the $1/k!$-fractional matching with $\varphi_e(h) = 1/|F_e|$ for each $h \in \{g(f(e))\} \cup F_e$ and $\varphi_e(h) = 0$ otherwise. Note that $\varphi_e$ has weight at least $1 + 1/(k-1)$. Let $\varphi_0$ be the $1/k!$-fractional matching induced by the matching $M' \setminus \{g(f(e)) \colon e \in M_4\}$. Note that the completion of $\varphi_0 + \sum_{e \in M_4} \varphi_e$ with respect to $B$ is a $1/k!$-fractional matching in $B$ of weight at least $t + k\delta n /(k-1) \geq t + \delta n$, so we are done.

    \item \textbf{\boldmath Every $e \in M_4$ satisfies \cref{e_property_2}.\unboldmath} Let $e \in M_4$ and let $i^* \in [k-1]$ be such that $g_{i^*}(e) \notin B$. Let $P^*$ be the tight pseudo-walk starting in $g(f(e))$ and then following $P_i(e)$ from $f(e)$ until $g_{i^*}(e)$. 
    Note that $V(P^*) \subseteq e \cup f(e) \cup g(f(e))$ and all edges of $P^*$ except its endpoints $g(f(e))$ and $g_{i^*}(e)$ are in $R$ (since they are all red by the definition of $g_{i^*}(e)$ and one of them is $f(e) \in R$). 
    Moreover, $g(f(e))$ and $g_{i^*}(e)$ are in different blue tight components of $H$. 
    For each $v \in f(e)$, construct a tight pseudo-walk $P_v$ from $g_{i^*}(e)$ to $g(f(e))$ such that $V(P_v) \subseteq g(f(e)) \cup g_{i^*}(e)$ and no edge of $P_v$ except $g(f(e))$ or $g_{i^*}(e)$ contains $v$.\footnote{It is easy to see that $P_v$ exists. Let $g(f(e)) \cap g_{i^*}(e) = u_1 \dots u_s$, $g_{i^*}(e) \setminus g(f(e)) = x_1 \dots x_{k-s}$, and $g(f(e)) \setminus g_{i^*}(e) = y_1 \dots y_{k-s}$. If $v = u_j$ for some $j \in [s]$, then we may take $P_v$ to be the tight pseudo-walk induced by the vertex sequence $vx_1 \dots x_{k-s-1}u_1 \dots u_{j-1} u_{j+1} \dots u_s y_1 \dots y_{k-s-1}v$. And if $v=x_j$ for some $j \in [k-s]$, then we may take $P_v$ to be the tight pseudo-walk induced by the vertex sequence $vx_1 \dots x_{j-1}x_{j+1}\dots x_{k-s}u_1 \dots u_s y_1 \dots y_{k-s}$. Finally, if $v = y_j$ for some $j \in [k-s]$, then we may take $P_v$ to be tight pseudo-walk induced by the vertex sequence $x_1 \dots x_{k-s}u_1 \dots u_s y_1 \dots y_{j-1} y_{j+1} \dots y_{k-s}v$.}
    By applying \cref{lem:blue_tight_walk} to the closed tight pseudo-walk obtained by the concatenation of $P^*$ and $P_v$, we deduce that $P_v$ contains an edge $f_v \in R$. 
    Let $F = \{f_v \colon v \in f(e)\}$ and note that $\bigcap (F \cup \{f(e)\}) = \varnothing$ since $v \notin f_v$ for each $v \in f(e)$. 
    Let $\varphi_e \colon R[e \cup f(e) \cup g(f(e))] \rightarrow [0,1]$ be the $1/k!$-fractional matching with $\varphi_e(h) = 1/|F|$ for each $h \in \{f(e)\} \cup F$ and $\varphi_e(h) = 0$ otherwise. 
    Note that $\varphi_e$ has weight at least $1 + 1/k$. Let $\varphi_0$ be the $1/k!$-fractional matching induced by the matching $M \setminus \{f(e) \colon e \in M_4\}$. 
    Now observe that the completion of $\varphi_0 + \sum_{e \in M_4} \varphi_e$ with respect to $R$ is a $1/k!$-fractional matching in $R$ of weight at least $t+ |M_4|/k = t + \delta n$, so we are done. \qedhere
    \end{enumerate}
\end{proof}

\subsection{\texorpdfstring{Combining \cref{lem:matching_inc_1} and \cref{lem:matching_inc_2}}{Combining the Main Lemmas}}

Finally, we combine \cref{lem:matching_inc_1} and \cref{lem:matching_inc_2} to derive the following lemma which states that any $(1-\eps, \eps)$-dense $2$-edge-coloured $k$-graph $H$ on $n$ vertices contains the desired fractional matchings. This lemma together with \cref{cor:matching_to_cycle} (a hypergraph version of {\L}uczak's connected matching method) then implies \cref{thm:k_unif_ramsey} and \cref{thm:k_unif_Lehel}. 

\begin{lemma} \label{lem:frac_matchings_exist}
    Let $1/n \ll \eps \ll \beta \ll \eta \ll 1/k$. Let $H$ be a $(1-\eps, \eps)$-dense $2$-edge-coloured $k$-graph on $n$ vertices. Then there exist fractional matchings $\varphi_1$ and $\varphi_2$ in $H$ and a red tight component $R$ and a blue tight component $B$ of $H$ such that the following hold.
    \begin{enumerate}[label = \upshape{(\roman*)}]
        \item $  \varphi_1$ has weight at least $(1-3\eta)n/(k+1)$ and $\varphi_2$ has weight at least $(1-3\eta)n/k$.
        \item All edges of non-zero weight for both $\varphi_1$ and $\varphi_2$ have weight at least $\beta$.
        \item We have $\{e \in H \colon \varphi_1(e) >0\} \subseteq R$ or $\{e \in H \colon \varphi_1(e) >0\} \subseteq B$. 
        \item We have $\{e \in H \colon \varphi_2(e) >0\} \subseteq R \cup B$.
    \end{enumerate}
\end{lemma}

Before proceeding with the details, we give a brief sketch of the proof.
We let $M_*$ be a matching in a monochromatic tight component of a blow-up $H_*$ of $H$ such that $M_*$ covers the largest fraction of the vertices possible. If $M_*$  is an almost perfect matching, then we are done immediately as we get the desired fractional matchings by converting $M_*$ into an almost perfect fractional matching in $H$ using \cref{prop:matching_to_fractional}. Thus we may assume that $M_*$ is not almost perfect. This allows us to apply \cref{lem:matching_inc_1}.
If \cref{event_red_matching} holds, we get a contradiction to the optimality of $M_*$ as we get a larger fractional matching in a monochromatic tight component which we convert into an integral matching in a blow-up by applying \cref{prop:matching_to_fractional}. Thus \cref{event_blue_matching} must hold. If $M_*$ does not cover at least roughly a $\frac{k}{k+1}$-fraction of the vertices, we now get a contradiction by also applying \cref{lem:matching_inc_2}. Thus $M_*$ covers a large enough fraction of the vertices to obtain the fractional matching needed for \cref{thm:k_unif_ramsey}. To also obtain the fractional matching needed for \cref{thm:k_unif_Lehel}, we then need to slightly modify the matching using the properties given by \cref{event_blue_matching}.

\begin{proof}[Proof of \cref{lem:frac_matchings_exist}]
    Let $r = k!$. Let $\gamma$ and $\delta$ be new constants with $\beta \ll \gamma \ll \delta \ll \eta \ll 1/k$.
    Let $L$ be the largest integer such that the $r^L$-blow-up $H_*$ of $H$ contains a matching $M_*$ contained in a monochromatic tight component $T_*$ of $H_*$ with $|M_*| \geq (\frac{1}{k^2 15^k} + L\gamma)|V(H_*)|$. Note that trivially $L  \leq \frac{1}{\gamma}\left(\frac{1}{k} - \frac{1}{k^2 15^k}\right) \leq 1/(k \gamma)$. By \cref{lem:large_MTC}, we have that $L \geq 0$. 

    If $L > \frac{1}{\gamma} (\frac{1-2\eta}{k} - \frac{1}{k^2 15^k})$, then $|M_*| \geq \frac{(1-2\eta)}{k}|V(H_*)|$.
    Let $\varphi_*$ the completion with respect to $H_*$ of the fractional matching induced by $M_*$.
    By \cref{prop:matching_to_fractional} with $r^L$, $1$, $k$, $H$, $H_*$, $\{T_*\}$, $\varphi_*$ playing the roles of $r$, $r'$, $k$, $H$, $H_*$, $\cC_*$, $\varphi_*$ respectively, we get that $H$ contains a $1/r^L$-fractional matching $\varphi$ with weight at least $(1-2\eta)n/k$ such that all its non-zero edges are contained in a single monochromatic tight component of $H$. Since $L \leq 1/(k\gamma)$, we have 
    \begin{align*}
        1/r^L \geq 1/r^{1/(k\gamma)} \geq \beta.
    \end{align*}
    Thus we are done since $\varphi$ satisfies the conditions for both $\varphi_1$ and $\varphi_2$.
    
    Hence we may assume that $L \leq \frac{1}{\gamma} (\frac{1-2\eta}{k} - \frac{1}{k^2 15^k})$.
    Let $M_*$ be a matching in the $r^L$-blow-up $H_*$ of $H$ such that $M_*$ is contained in the monochromatic tight component $T_*$ of $H_*$ and $|M_*| = (\frac{1}{k^2 15^k} + L\gamma)|V(H_*)| \leq (1-\eta)\frac{|V(H_*)|}{k}$. Assume without loss of generality that $T_*$ is a red tight component $R_*$. Since $H$ is $(1-\eps,\eps)$-dense, by \cref{prop:H_*_is_dense}, $H_{*}$ is $(1-2\eps, 2\eps)$-dense.
    We apply \cref{lem:matching_inc_1}, with $|V(H_*)|$, $2\eps$, $\gamma$, $\eta$, $k$, $H_*$, $M_*$, $R_*$, playing the roles of $n$, $\eps$, $\gamma$, $\eta$, $k$, $H$, $M$, $R$, respectively, to obtain that at least one of the following holds. 
    \begin{enumerate}[label = \upshape{(MB\arabic*)}, leftmargin= \widthof{M100000}]
        \item There is a $1/r$-fractional matching in $R_*$ of weight at least $|M_*| + \gamma |V(H_*)|$. \label{event_red_matching}
        \item There is a matching $M'$ in a blue tight component $B_*$ of $H_*$ with 
        \[
            |M'| \geq \min(|M_*|, |V(H_*)| - |V(M_*)|) - \sqrt{\gamma} |V(H_*)|,
        \]
        such that each edge $e \in M'$ intersects exactly one edge $f(e) \in M_*$ and $|e \cap f(e)| = k-1$ and $H_*[e \cup f(e)] \cong K_{k+1}^{(k)}$. \label{event_blue_matching}
    \end{enumerate}
    Note that $|M_*| + \gamma |V(H_*)| = (\frac{1}{k^2 15^k} + (L+1)\gamma)|V(H_*)|$.
    If \ref{event_red_matching} holds, then by \cref{prop:matching_to_fractional}, we have that in the $r$-blow-up $H_{**}$ of $H_*$ there is a matching $M_{**}$ in a red tight component of $H_{**}$ with $|M_{**}| \geq (\frac{1}{k^2 15^k} + (L+1)\gamma)|V(H_{**})|$, a contradiction to the maximality of $L$, since $H_{**}$ is an $r^{L+1}$-blow-up of $H$. So \ref{event_blue_matching} holds. We distinguish the following two cases.

    \begin{enumerate}[label=\textbf{Case \Alph*:}, ref=\Alph*, wide, labelwidth=0pt, labelindent=0pt]
    \item \textbf{\boldmath $L \leq \frac{1}{\gamma}(\frac{1-2\eta}{k+1} -  \frac{1}{k^2 15^k})$.\unboldmath}
    In this case, we have that 
    \begin{align*}
        |M_*| &= \left(\frac{1}{k^2 15^k} + L\gamma\right)|V(H_*)| \leq (1-\eta)\frac{|V(H_*)|}{k+1} \leq |V(H_*)| - \frac{k}{k+1}|V(H_*)| \\ 
        &\leq |V(H_*)| - |V(M_*)|.
    \end{align*}
    So 
    \begin{align*}
        |M'| \geq \min(|M_*|, |V(H_*)| - |V(M_*)|) - \sqrt{\gamma} |V(H_*)| = |M_*| - \sqrt{\gamma} |V(H_*)|.
    \end{align*}
    By \cref{lem:matching_inc_2}, we have that there exists a $1/r$-fractional matching in $R_*$ or $B_*$ of weight at least 
    \begin{align*}
        |M'| + \delta |V(H_*)| &\geq |M_*| - \sqrt{\gamma} |V(H_*)| + \delta |V(H_*)| \geq |M_*| +\gamma |V(H_*)| \\ 
        &= \left(\frac{1}{k^2 15^k} + (L+1)\gamma\right)|V(H_*)|
    \end{align*}
    By \cref{prop:matching_to_fractional}, we have that in the $r$-blow-up $H_{**}$ of $H_*$ there is a matching $M_{**}$ in a monochromatic tight component of $H_{**}$ with 
    \begin{align*}
        |M_{**}| \geq \left(\frac{1}{k^2 15^k} + (L+1)\gamma\right)|V(H_{**})|,
    \end{align*}
    a contradiction to the maximality of $L$, since $H_{**}$ is an $r^{L+1}$-blow-up of $H$.
    \item \textbf{\boldmath $L > \frac{1}{\gamma}(\frac{1-2\eta}{k+1} -  \frac{1}{k^2 15^k})$.\unboldmath}
    Recall that $L \leq \frac{1}{\gamma}(\frac{1-2\eta}{k} -  \frac{1}{k^2 15^k})$.
    Hence, we have that 
    \begin{align} \label{size_M_*_bounds}
        (1-2\eta)\frac{|V(H_*)|}{k+1} \leq |M_*| \leq (1-\eta)\frac{|V(H_*)|}{k}.
    \end{align}
    By \cref{prop:matching_to_fractional}, we have that $\varphi_1$ exists. We now show that $\varphi_2$ exists. We define a $1/r$-fractional matching $\widetilde{\varphi}_2$ as follows. Let $M_0 = \{f(e) \colon e \in M'\}$ and
    \begin{align*}
        \widetilde{\varphi}_2(f) \coloneqq 
        \begin{cases}
            1/k &\text{ if $f \in H[e \cup f(e)]$ for some $e \in M'$,} \\
            1 &\text{ if $f \in M_* \setminus M_0$,} \\
            0 &\text{ otherwise}.
        \end{cases}
    \end{align*}
    Note that $M_* \setminus M_0 \subseteq R_*$. Moreover, note that for each $e \in M'$, $e \cup f(e)$ is a set of size $k+1$ that contains the red edge $f(e) \in R_*$ and the blue edge $e \in B_*$. Hence, for each $e \in M'$, $H[e \cup f(e)] \subseteq R_* \cup B_*$. 
    Thus $\widetilde{\varphi}_2$ gives non-zero weight only to edges in $R_* \cup B_*$.
    By \cref{event_blue_matching}, for every $e \in M'$, we have $H[e \cup f(e)] \cong K_{k+1}^{(k)}$. It follows that $\widetilde{\varphi}_2$ has weight $\frac{k+1}{k}|M'| + |M_*| - |M_0| = |M_*| + |M'|/k$.
    By \cref{size_M_*_bounds}, we have
    \begin{align*}
        \begin{split}
            |M_*| + |M'|/k &\geq |M_*| + \frac{1}{k}\min(|M_*|, |V(H_*)| - |V(M_*)|) - \sqrt{\gamma} \frac{|V(H_*)|}{k}  \\
            &= \min\left(\frac{k+1}{k} |M_*|, |V(H_*)| - (k-1) |M_*|\right) - \sqrt{\gamma} \frac{|V(H_*)|}{k}  \\
           &\geq \min\left((1-2\eta)\frac{|V(H_*)|}{k}, (1-\eta)\frac{|V(H_*)|}{k}\right) - \sqrt{\gamma} \frac{|V(H_*)|}{k} \\
           &\geq (1-3\eta) \frac{|V(H_*)|}{k}.
       \end{split}
    \end{align*}
    Thus $\widetilde{\varphi}_2$ has weight $(1-3\eta) \frac{|V(H_*)|}{k}$. Hence by \cref{prop:matching_to_fractional} with $r$, $k$, $k$, $H$, $H_*$, $\{R_*, B_*\}$, $\widetilde{\varphi}_2$ playing the roles of $r$, $r'$, $k$, $H$, $H_*$, $\cC_*$, $\varphi_*$, respectively, we have that $\varphi_2$ exists. \qedhere
    \end{enumerate}
\end{proof}

\section{Proof of the Main Theorems}

\cref{thm:k_unif_ramsey,thm:k_unif_Lehel} now follow straightforwardly from {\cite[Corollary 20]{Lo2020}} (a hypergraph version of {\L}uczak's connected matching method). To state this result, we need the following definition.

\begin{definition}[{\cite[Definition 19]{Lo2020}}]
Let $\mu_k^1(\beta,\eps, n)$ be the largest~$\mu$ such that every $2$-edge-coloured $(1-\eps, \eps)$-dense $k$-graph on~$n$ vertices contains a fractional matching with weight~$\mu$ such that all edges with non-zero weight have weight at least~$\beta$ and lie in a single monochromatic tight component.
Let $\mu_k^1(\beta) = \liminf_{\eps \to 0} \liminf_{n \to \infty} \mu_k^1(\beta,\eps,n)/n$.
Similarly, let $\mu_k^*(\beta,\eps,n)$ be the largest~$\mu$ such that every $2$-edge-coloured $(1 - \eps, \eps)$-dense $k$-graph on~$n$ vertices contains a fractional matching with weight~$\mu$ such that all edges with non-zero weight have weight at least~$\beta$ and lie in one red and one blue tight component.
Let $\mu_k^*(\beta) = \liminf_{\eps \to 0} \liminf_{n \to \infty} \mu_k^*(\beta,\eps,n)/n$. 
\end{definition}

We now state the hypergraph version of {\L}uczak's connected matching method that we use.

\begin{corollary}[{\cite[Corollary 20]{Lo2020}}] \label{cor:matching_to_cycle}
    Let $1/n \ll \eta, \beta, 1/k$ with $k \ge 3$. 
    Let~$K$ be a $2$-edge-coloured complete $k$-graph on~$n$ vertices.
    Then the following hold.
    \begin{enumerate}[label= \upshape{(\roman*)}]
        \item $K$ contains two vertex-disjoint monochromatic tight cycles of distinct colours covering at least $(\mu_k^*(\beta) - \eta)k n$ vertices and
        \item $K$ contains a monochromatic tight cycle of length~$\ell$ for any $\ell \leq (\mu_k^1(\beta) - \eta)k n$ divisible by~$k$.
    \end{enumerate}
\end{corollary}

We are now ready to prove \cref{thm:k_unif_ramsey,thm:k_unif_Lehel}.

\begin{proof}[Proof of \cref{thm:k_unif_ramsey}]
    Let $1/n \ll \beta \ll \eta \ll \eps, 1/k < 1$ with $k \geq 3$. Let $K$ be a $2$-coloured complete $k$-graph on $N = (k+1+\eps)n$ vertices. We show that $K$ contains a monochromatic tight cycle of length $kn$. Note that by \cref{lem:frac_matchings_exist}, we have that $\mu_k^1(\beta) \geq (1-3\eta)/(k+1)$. By \cref{cor:matching_to_cycle}, we have that $K$ contains a monochromatic tight cycle of length $\ell$ for any $\ell \leq (\mu_k^1(\beta)-\eta)kN$ divisible by $k$. Since
    \begin{align*}
        (\mu_k^1(\beta)-\eta)kN \geq (1/(k+1) - 4\eta)(k+1 + \eps)kn \geq kn,
    \end{align*}
    $K$ contains a monochromatic tight cycle of length $kn$, as desired.
\end{proof}

\begin{proof}[Proof of \cref{thm:k_unif_Lehel}]
    Let $1/n \ll \beta \ll \eta \ll \eps, 1/k < 1$ with $k \geq 3$. Let $K$ be a $2$-coloured complete $k$-graph on $n$ vertices. We show that $K$ contains a red and a blue tight cycle that are vertex-disjoint and together cover at least $(1-\eps)n$ vertices. Note that by \cref{lem:frac_matchings_exist}, we have that $\mu_k^*(\beta) \geq (1-3\eta)/k$. By \cref{cor:matching_to_cycle}, we have that $K$ contains a red and a blue tight cycle that are vertex-disjoint and together cover at least $(\mu_k^*(\beta) - \eta)k n$ vertices.
    Since 
    \begin{align*}
        (\mu_k^*(\beta) - \eta)k n \geq ((1-3\eta)/k - \eta)k n \geq (1-\eps)n,
    \end{align*}
    we are done.
\end{proof}

\section{Concluding Remarks}
In this paper, we have proved the $i = 0$ case of \cref{HLPRRS_conjecture}. For the remaining cases, it is no longer sufficient to just find a large enough (fractional) matching in a monochromatic tight component of the reduced $k$-graph. Indeed, for these cases, we need some way of adjusting the parity\footnote{By \emph{parity}, we mean which residue class modulo $k$ the length of the tight cycle belongs to.} of the length of the tight cycle we obtain in the original $k$-graph. The natural approach for doing this is to find in the monochromatic tight component not just a large enough (fractional) matching, but also a gadget that allows one to change the parity of the resulting tight cycle. Such a gadget could, for example, be a short tight cycle of an appropriate length. Indeed, this approach was successfully implemented for $3$-graphs in~\cite{Haxell2009,Haxell2007}. Some of our ideas here may be useful to proving the remaining cases of \cref{HLPRRS_conjecture}, however, a lot of our arguments revolve around the fact that larger and larger matchings get us closer to our goal. This is no longer the case if $i > 0$, as for those cases, a large fractional matching in a monochromatic tight component of the reduced $k$-graph is no longer enough if we cannot also find a gadget to fix the parity.

We also proved an approximate Lehel's conjecture for $k$-uniform tight cycles. Indeed, we showed that, for $k \geq 3$, every $2$-edge-coloured complete $k$-graph contains a red and a blue tight cycle that are vertex-disjoint and together cover $n - o(n)$ vertices. Due to a counterexample (see \cite[Proposition 1]{Lo2020}), we know that the $o(n)$-term cannot be replaced by $0$.
In \cite{Garbe2024}, it was proved that in the case of $k = 3$, the $o(n)$-term can be replaced by a constant. 
We conjecture that it can also be replaced by a constant for all $k \geq 4$.
\begin{conjecture}
    For every $k \geq 3$, there exist $c_k \geq 1$ and $n_0(k) \in \mathbb{N}$ such that for all $n \geq n_0(k)$ the following holds. Every $2$-edge-coloured complete $k$-graph on $n$ vertices contains a red tight cycle and a blue tight cycle that are vertex-disjoint and together cover at least $n - c_k$ vertices.
\end{conjecture}
We remark that allowing the two tight cycles to possibly have the same colour might allow us to cover more vertices (as has been shown to be the case for $k = 3$ in \cite{Garbe2024}).

\section*{Acknowledgements}

The author thanks Alexandru Malekshahian for helpful discussions. The author also thanks two anonymous referees whose comments helped improve the exposition of the paper. 

This project has received partial funding from the European Research Council (ERC) under the European Union's Horizon 2020 research and innovation programme (grant agreement no.\ 786198) and was also supported in part by the Austrian Science Fund (FWF) [10.55776/I6502].

For the purpose of open access, the author has applied a CC BY public copyright licence to any Author Accepted Manuscript version arising from this submission.  

\bibliographystyle{abbrv}
\bibliography{New_bibliography}

\end{document}